\documentclass[letterpaper,12pt]{article}
\newtheorem{theorem}{Theorem}[section]

\newtheorem{corollary}[theorem]{Corollary}

\newenvironment{proof}[1][Proof]{\begin{trivlist}
\item[\hskip \labelsep {\bfseries #1}]}{\end{trivlist}}

\newcommand{\qed}{\nobreak \ifvmode \relax \else
      \ifdim\lastskip<1.5em \hskip-\lastskip
      \hskip1.5em plus0em minus0.5em \fi \nobreak
      \vrule height0.75em width0.5em depth0.25em\fi}

\usepackage{amsmath,amssymb}
\numberwithin{equation}{section}
\title{\bf\large
Characterization of Compact Subsets of $\mathcal{A}^p$ with Respect to Weak Topology}
\author{Hirbod Assa\thanks{{\tt e-mail: assa@dms.umontreal.ca}\,\,,\,\,\,Department of Mathematics and Statistics , University of Montreal,CP 6128, succ. Centre-ville 
Montr\'eal, Qu\'ebec H3C 3J7 
Canada}\\
\small Department of Mathematics and Statistics\\
\small University of Montreal}
\date{}

\begin{document}

\maketitle

\abstract{In this brief article we characterize the relatively compact subsets of $\mathcal{A}^p$ for the topology  $\sigma(\mathcal{A}^p,\mathcal{R}^q)$ (see below), by the weak compact subsets of $L^p$ . The spaces $\mathcal{R}^q$ endowed with the weak topology induced by $\mathcal{A}^p$, was recently employed to create the convex risk theory of random processes. The weak compact sets of $\mathcal{A}^p$ are important to characterize the so-called Lebesgue property of convex risk measures, to give a complete description of the Makcey topology on $\mathcal{R}^q$ and for their use in the optimization theory.}

\section{Introduction}
Let $(\Omega,\mathcal{F},\mathbb{P})$ be a standard topological space and let $(\mathcal{F}_t)_{0\leq t\leq T}$ be a filtration with the usual conditions.
\\The set of c\`adl\`ag processes is a big family of random processes  which for example contains the space of martingales, submartingales and supermartingales. This set is endowed with a so-called Skorokhod metric for which the induced topology is complete and separable. Furthermore the tight sequences and the relation between this topology and weak convergence (in the probabilistic context) are very well known \cite{Ja}. Although these are nice properties, the big pitfall associated with this topology is that the operator of finite summation is not continuous. This fact indicates that this space can not be considered as a topological vector space. Having the topological vector space structure is important not only to take the advantage of the theory of topological vector spaces, but also to define the convex functions and to study the smoothness of the convex functions with respect to the topology of the space \cite{Au} ,\cite{Rc}.
\\On the other hand there are subsets of the set of c\`adl\`ag processes including all continuously differentiable processes, Brownian motion and Levy processes and also have very nice topological vector space structures \cite{DM2},\cite{Ja},\cite{Co} and \cite{Ka}. As an  example:
\begin{equation}
 \mathcal{H}^1=\left \{H:[0,T]\times\Omega\longrightarrow \mathbb{R}\Bigg |
\begin{array}{clcr}
& H \,\,\text{is c\`adl\`ag} \\
& H \,\,\text{adapted and martingale}\\
& (H)^*\in L^1
\end{array}
 \right\} ,
\end{equation}
where $(H)^*(\omega)=\sup\limits_{0\leq t\leq T}\vert H_t(\omega)\vert$. The space $\mathcal{H}^1$ is a Banach space with the norm $\Vert H\Vert=\Vert(H)^*\Vert_{L^1}.$ This space as a Banach space and its dual which is BMO have been studied in many earlier works(for example see \cite{DM2}). The compact subsets of this space with respect to weak topology $\sigma(\mathcal{H}^1,BMO)$ have been characterized in \cite{DMY}.
The general subject is the weak topology $\sigma(\mathcal{H}^1,BMO)$ on the space $\mathcal{H}^1$. Its relatively compact sets are characterized by a uniform integrability property of the maximal functions. 
\\In our work we aim to characterize the compact sets of a much bigger family of random processes $\mathcal{R}^q$ which contained $\mathcal{H}^1$. The space $\mathcal{R}^q$ is defined as follows:
\begin{equation} 
 \mathcal{R}^q=\left \{X:[0,T]\times\Omega\longrightarrow \mathbb{R}\Bigg |
\begin{array}{clcr}
& X \,\,\text{is c\`adl\`ag} \\
& X \,\,\text{adapted}\\
& (X)^*\in L^q
\end{array}
 \right\} ,
\end{equation} 
 This space is a Banach space with the norm $\Vert X\Vert_{\mathcal{R}^q}=\Vert(X)^*\Vert_{L^q}.$
\\Among the subsets of c\`adl\`ag processes, the spaces $\mathcal{R}^q$ for $q\in[1,+\infty]$  are not too far from the set of c\`adl\`ag processes and they also have nice topological vector space structures. Recently the authors of \cite{Ch} and \cite{Ch2} have employed the family $\mathcal{R}^q$  to develop the theory of convex risk measures for the weak topology induced by $\mathcal{A}^p$ (see below). By $\mathcal{A}^p$ we mean:
\begin{equation}
\left\{a:[0,T]\times\Omega\longrightarrow \mathbb{R}^2\Bigg |
\begin{array}{clcr}
&a=(a^{\text{pr}},a^{\text{op}}),\\ 
&a^{\text{pr}},a^{\text{op}} \,\text{right continuous of finite  variation}\\
&a^{\text{pr}} \text{predictable},a^{\text{pr}}_0=0\\
&a^{\text{op}}\text{optional , purely\,discontinuous}\\
&\text{Var}(a^{\text{pr}})+\text{Var}(a^{\text{op}})\in L^p
\end{array}
 \right\} \;, 
\end{equation}
where $\text{Var}(f)$ is the variation of function $f$. This space is a Banach space with the norm $\Vert a\Vert_{\mathcal{A}^p}=\Vert \text{Var(a)}\Vert_{L^p}$. For $\frac{1}{p}+\frac{1}{q}=1$ the dual relation between $\mathcal{A}^p$ and $\mathcal{R}^q$ is as follows:
\begin{equation}
\langle X,a\rangle=E[(X|a)],
\end{equation}
where $(X|a)=\int\limits_{]0,T]}X_{t-}da_t^{\text{pr}}+\int\limits_{[0,T]} X_tda_t^{\text{op}}.$ 
\\The compact sets of the space $\mathcal{A}^p$ are important for several reasons. In risk theory they can characterize the convex risk measures with the Lebesgue property(see \cite{As}, \cite{Jo}). In topological vector spaces theory, they can give a complete description of Makcey's topology on $\mathcal{R}^q$ (see \cite{Gr}) and in optimization theory they can be used to optimize the lower semi continuous convex functions  (see \cite{Au}, \cite{Rc}).
 \\In this paper we study the compact subset of $\mathcal{A}^p$ for the topology $\sigma(\mathcal{A}^p,\mathcal{R}^q)$ where $\frac{1}{p}+\frac{1}{q}=1$. We characterize the compactness in $\mathcal{A}^p$ by the compactness (and in the case $p=1$ by the uniform integrability) of the variation functions in $L^p$. 
\\The paper is organized as follows: in Section 2 we give some definitions and remarks for which we need to prove our main result. In Section 3 we characterize the compact sets of $\mathcal{A}^p$ by the compact sets of $L^p$ for the weak topology.

\section{Some definitions and Remarks}
In this section we give some definitions that are needed in Section 3.

 Consider $\hat {\mathcal{F}_t}=\mathcal{F}$ and $(\hat {\mathcal{R}}^q,\hat {\mathcal{A}}^p)$, the corresponding process spaces. Let $\Pi^{\text{op}}$ and $\Pi^{\text{pr}}$ be the optional and predictable projections. We also give the dual optional and predictable projection of finite variation processes with the same notation (as references see \cite{Ch},\cite{DM2} and \cite{Ka}). We define the projection $\Pi^*:\hat{\mathcal{A}}^p\rightarrow {\mathcal{A}^p}$ as follows : let $a=(a^l,a^r)\in\hat{\mathcal{A}}^p$. Let $\tilde a^l=\Pi^{\text{pr}}(a^l)$ and $\tilde a^r=\Pi^{\text{op}}(a^r)$. Then one can split $\tilde a^r$ uniquely into a purely discontinuous finite variation part $\tilde a^r_d$ and a continuous finite variation part $\tilde a^r_c$ with $\tilde a^r_c(0)=0$. Now define:
$$\Pi^*(a)=(\tilde a^l+\tilde a^r_c,\tilde a^r_d).$$
 We know that every predictable process is also optional so $\tilde a^l,\tilde a^r_c,\tilde a^r_d$  are all optional. This fact by  definition of $\Pi^*$ give that for every $X\in\mathcal{R}^q$ we have:
\begin{equation}
\label{3.0}
\langle X,a\rangle=\langle X,\Pi^*(a)\rangle.
\end{equation}
For more details see  relation 3.5 , Remark 3.6 \cite{Ch}.
\\Let $a=(a^{\text{pr}},a^{\text{op}})\in{\mathcal{A}}^p$. Since any predictable process is optional then by Theorem $2.1.53$ \cite{Ka} the measure $\mu(A)=\langle 1_A,a\rangle$ is optional and then $\langle X,a\rangle=\langle \Pi^{\text{op}}(X),a\rangle.$ This relation with \ref{3.0} yields that $\forall X\in\hat{\mathcal{R}}^q,a\in\hat{\mathcal{A}^p}$:
\begin{equation}
\label{3.0.1}
\langle \Pi^{\text{op}}(X),a\rangle=\langle \Pi^{\text{op}}(X),\Pi^*(a)\rangle=\langle X,\Pi^*(a)\rangle.
\end{equation}  
For more details the reader is referred to \cite{Ch}, Remarks 2.1 , 3.6 and \cite{Ka}.

For every random variable $X\in L^q(\Omega,\mathcal{F})$ , we identify the random variable $X$ and the random process $X_t:=X\,\,\,,\,\,\forall t\in[0,T]$.
\\\textbf{Remark}{\bf 2.1}: Relation \ref{3.0} (or \ref{3.0.1}) shows that $\Pi^*$ is $\sigma(\hat{\mathcal{A}}^p,\hat{\mathcal{R}}^q)$/$\sigma({\mathcal{A}}^p,{\mathcal{R}}^q)$ is continuous.
\\\textbf{Remark}{\bf 2.2}: Let $X\in L^q(\Omega)$ be a random variable. By Doob's Stopping Theorem
 it is easy to see that the optional projection of constant random process $X$ is the martingale $M_t:=E[X|\mathcal{F}_t]$.
 So then $\forall X\in L^q, a\in\hat{\mathcal{A}}^p$:
\begin{equation}
\label{3.2}
E[(a_T-a_0)X]=\langle X,a\rangle=\langle \Pi^{\text{op}}(X),a\rangle\,,\,\,
\end{equation}
\textbf{Remark}{\bf {2.3}}: For any uniformly integrable (or weakly relatively compact ) subset $A$ of $L^p$ and a bounded sequence $X_n\in L^q$ converging in probability to $X\in L^q$ we have:
\begin{equation}\label{2.5}
 \sup\limits_{f\in A}E[fX_n]\rightarrow\sup\limits_{f\in A}E[fX]
\end{equation}
For that see \cite{Jo}.
\\\textbf{Remark}{\bf {2.4}}: Any member $a\in\mathcal{A}^p$ can be written as the difference of two increasing parts $a^+-a^-$ for which almost surely $a^+$ and $a^-$ as measures on $[0,T]$ have disjoint supports.

\section{Characterization of the Compact Sets of $\mathcal{A}^p$}

In this section we give the characterization of the compact sets of $\mathcal{A}^p$ with respect to the compact sets of $L^p$. By r.c. we mean relatively compact.
\begin{theorem}
Let $A\subset\mathcal{A}^p$ and $\frac{1}{p}+\frac{1}{q}=1$. The following three conditions are equivalent:
\\1-$A$ is r.c. in the topology $\sigma(\mathcal{A}^p,\mathcal{R}^q)$. 
\\2-$\text{Var}(A)$ is r.c. in the topology $\sigma(L^p,L^q)$.
\\3-$C:=\{a_T-a_0|a\in A\}$ is r.c. in the topology $\sigma(L^p,L^q)$.
\end{theorem}

\begin{proof}(2)$\Leftrightarrow $(3). First of all we mention that a subset of $L^p$ is r.c. iff its absolute value is r.c. Actually for $p\neq 1$ this comes from the fact that bounded sets are r.c. sets and for $p=1$, by Dunford-Pettis theorem, uniformly integrable sets are r.c. sets. Let $A_\pm =\{a^\pm |a\in A\}$ where $a^+,a^-$ are increasing decomposition of $a$. It is obvious that $C_\pm =\{(a_T-a_0)^\pm | a\in A\}=\text{Var}(A_\pm )$. So we have $\text{Var}(A)\subseteq2\vert C\vert$ and $C\subseteq2\vert\text{Var}(A)\vert$. Now by above arguments the proof is complete.
\\ 
\\(1)$\Leftrightarrow $(2). We split this part into two cases. 
\\\textbf{Case 1: $p\neq 1$}.
\\Consider $\mathcal{F}_t=\mathcal{F}$. In this case by Theorems 65,67 of Section VII \cite{DM2} we know that $\mathcal{A}^p$ is the dual of $\mathcal{R}^q$. Since $\mathcal{A}^p$ is endowed with the weak* topology then $A$ is r.c. iff it is bounded and this is true iff $\text{Var}(A)$ is bounded or, in the other words, r.c. for topology $\sigma(L^p,L^q)$.
\\When $\mathcal{F}_t$ is nontrivial $A$ is a relatively compact set of $\hat{\mathcal{A}}^p$. The assertion is true because of the continuity of $\Pi^*$.
\\\textbf{Case 2: $p=1$}.
\\($\Rightarrow $):  We claim that $C_\pm $ are relatively compact. Let $a_T^\lambda-a_0^\lambda$ be a net in $C$ and $X$ be a member of $L^\infty$ . Then by the relative compactness of $A$ there is a subnet $a^\beta$ and $a$ such that $a_\beta\xrightarrow{\sigma(\mathcal{A}^1,\mathcal{R}^\infty)} a$. This gives:
 $$E[(a_T^\beta-a_0^\beta)X]=\langle \Pi^{\text{op}}(X),a^\beta\rangle\rightarrow \langle\Pi^{\text{op}}(X),a\rangle=E[(a_T-a_0)X].$$
That means $C$ is r.c. for $\sigma(L^1,L^\infty)$. By Dunford-Pettis theorem we know that this is equivalent to saying that $C$ is uniformly integrable. Then $\vert C\vert=\{\vert f\vert|f\in C\}$ is uniformly integrable and consequently $C_\pm $ are uniformly integrable. Again by Dunford-Pettis $C_\pm $ are r.c.
\\Now by $\text{Var}(A)\subseteq\text{Var}(A_+)+\text{Var}(A_-)=C_++C_-$ we get that $\text{Var}(A)$ is r.c.  
\\($\Leftarrow $):We define a topology on $\mathcal{R}^\infty$. For that we define the semi norms which generate this topology.
\\For any weakly relatively compact subset $H$ in $L^1$ let $V(H):=\{a\in\mathcal{A}^1|\exists f\in H\,\,\text{s.t.}\text{Var}(a)\leq\vert f\vert\}$. Now define the following semi norm for $H$ on $\mathcal{R}^\infty$:
$$P_H(X)=\sup\limits_{a\in V(H)}\langle X,a\rangle.$$
This topology is compatible with the vector structure because obviously the $V(H)$'s are bounded. We show this topology by $\sigma^1$. Let $(\mathcal{R}^\infty)^{'}$ be the dual of $\mathcal{R}^\infty$ with respect to topology $\sigma^1$. It is clear that $\mathcal{A}^1\subseteq (\mathcal{R}^\infty)^{'}$. We want to show that $\mathcal{A}^1= (\mathcal{R}^\infty)^{'}$. 
\\Let $\mu$ be an arbitrarily element of $(\mathcal{R}^\infty)'$ and $X_n$ be a non-negative sequence such that $(X_n)^*\xrightarrow{\mathbb{P}}0 $. Then by relation \ref{2.5} we have :
\begin{equation}
\label{new}
0\leq P_H(X_n)\leq\sup\limits_{f\in H} E[(X_n)^*\vert f\vert]\rightarrow 0.
\end{equation} 
This gives  $X_n\xrightarrow{\sigma^1} 0$  and then $\mu(X_n)\rightarrow 0$. This fact,and (5.1) of Chapter VII \cite{DM2} show that any $\mu$ can be decomposed into a difference of two positive functionals. Let $\mu^+$ be the positive part. By definition for any $X\geq0$ , $\mu^+(X)=\sup\limits_{0\leq Y\leq X}\mu(Y)$. Let $X_n$ be a positive and decreasing sequence for which $(X_n)^*\downarrow 0$ in probability. Let $0\leq Y_n\leq X_n$ be such that $\mu^+(X_n)\leq \mu(Y_n)+\frac{1}{n}$. Then since $(Y_n)^*\xrightarrow{\mathbb{P}}0 $ by \ref{new} we get: 
$$0\leq\mu^+(X_n)\leq\mu(Y_n)+\frac{1}{n}\rightarrow 0.$$
\\By this fact and Theorem 2 of Chapter VII \cite{DM2} we get that $\mu^+\in\mathcal{A}^1$. Similarly $\mu^-\in\mathcal{A}^1$ so then $\mu\in\mathcal{A}^1$.That means $\mathcal{A}^1=(\mathcal{R}^\infty)'$.
\\The Corollary to Mackey's Theorem 9, Section 13, Chapter 2 \cite{Gr} leads us to $\sigma^1\subseteq\tau(\mathcal{R}^\infty,\mathcal{A}^1)$, where $\tau(\mathcal{R}^\infty,\mathcal{A}^1)$ is the Mackey's topology. By this relation we get that for a relatively weakly compact subset $H$ of $L^1$ there exists $C$, a compact disk in $(\mathcal{A}^1,\sigma(\mathcal{A}^1,\mathcal{R}^\infty))$, for which $\{X| \sup\limits_{a\in C} \langle X,a\rangle< 1\}\subset\{X|P_H(X)\leq1\}$. By polarity $V(H)\subseteq\{X|P_H(X)\leq1\}^\circ\subseteq \{X| \sup\limits_{a\in C} \langle X,a\rangle< 1\}^\circ $. Using the generalized Bourbaki-Alaoglu Theorem we get that $\{X| \sup\limits_{a\in C} \langle X,a\rangle< 1\}^\circ$ is compact in the topology $\sigma(\mathcal{A}^1,\mathcal{R}^\infty)$.
\\Let $H=\text{Var}(A)$. By $A\subseteq V(\text{Var}(A))$,  the proof is complete.  
\\
$\square$
\end{proof}
The following two corollaries are immediate:
\begin{corollary}
Let $A\subseteq\mathcal{A}^1$. Then $A$ is $\sigma(\mathcal{A}^1,\mathcal{R}^\infty)-$ r.c. iff $\text{Var}(A)$ is uniformly integrable.
\end{corollary}
\begin{corollary}
The set $A\subseteq\mathcal{A}^p$, for $p\neq\infty$, is $\sigma(\mathcal{A}^p,\mathcal{R}^q)-$ r.c. iff it is sequentially r.c. 
\end{corollary}


\begin{thebibliography}{amsplain}
\bibitem{As} Assa ,H \textsl{Lebesgue property for convex risk measures on bounded c\'adl\'ag processes and its application} Preprint, Submitted to the Journal of Stochastic Processes and Their Applications.
\bibitem{Au} Aubin, J.P. (1998). Optima and Equilibria Graduate Texts in Mathematics. Springer-
Verlag.
\bibitem{Ch}Cheridito, P; Delbaen, F; Kupper, M \textsl{Coherent and Convex Monetary Risk Measures for Bounded c\'dl\'ag processes},Stochastic Process.Appl,Vol 112,(2004), no 1,1--22.
\bibitem{Ch2}Cheridito, P; Delbaen, F; Kupper, M,\textsl{Coherent and convex monetary risk measures for unbounded c\'adl\'ag processes.Finance and Stochastics}, Vol IX, (2005), issue 3.
\bibitem{De} Delbaen, Freddy \textsl{Coherent risk measures on general probability spaces.  Advances in finance and stochastics},  1--37, Springer, Berlin, 2002.
\bibitem{DM2}Dellacherie, C., Meyer, P.A., 1982. \textsl{Probabilities and Potential} B. North-Holland ,Amsterdam (Chapter V-VIII).
\bibitem{DMY}Dellacherie, C; Meyer, P-A; Yor, M \textsl{Sur certaines propri\'et\'es des espaces de Banach $\mathcal{H}^1$ et $BMO$} S\'eminaire de probabilit\'es de Strasbourg, 12 (1978), p. 98-113 
\bibitem{Gr} Grothendieck, A. (1973). Topological Vector Spaces. Gordon and Breach, New York.
\bibitem{Ja}Jacod,J; Shiryaev, A.N(Berlin : Springer, c2003.)\textsl{Limit Theorems for Stochastic Processes }
\bibitem {Jo} Jouini, Elyès; Schachermayer, Walter; Touzi, Nizar \textsl{Law invariant risk measures have the Fatou property.}  Advances in mathematical economics. Vol. 9,  49--71
\bibitem{Ka} (Edited by) Kannan, D; Lakshmikantham, V \textsl{Handbook of Stochastic Analysis and Applications}. New York , Marcel Dekker, Inc (2002).
\bibitem{Rc} Rockafellar, R.T , 1996.\textsl{Convex Analysis} Princeton Landmarks in Mathematics.



\end{thebibliography}
\end{document}